\documentclass[12pt, a4paper]{article}

\usepackage[utf8]{inputenc}
\usepackage[T1]{fontenc}
\usepackage[english]{babel}

\usepackage[top=2.5cm, bottom=2.5cm, left=2.5cm, right=2.5cm]{geometry}
\usepackage{setspace}
\usepackage{amsmath}
\usepackage{amssymb}
\usepackage{amsthm}
\usepackage{mathtools}

\theoremstyle{plain}
\newtheorem{theoreme}{Theorem}[section]

\newtheorem{proposition}[theoreme]{Proposition}
\newtheorem{corollaire}[theoreme]{Corollary}

\theoremstyle{definition}
\newtheorem{definition}[theoreme]{Definition}
\newtheorem{exemple}[theoreme]{Example}

\theoremstyle{remark}
\newtheorem{remarque}[theoreme]{Remark}

\usepackage{graphicx}
\usepackage{float}
\usepackage{caption}
\usepackage{subcaption}
\usepackage{cite}

\usepackage{booktabs}
\usepackage{array}

\usepackage{hyperref}
\hypersetup{
    colorlinks = true,
    linkcolor  = blue!70!black,
    citecolor  = green!50!black,
    urlcolor   = cyan!70!black,
}
\usepackage[english]{cleveref}

\usepackage[numbers, sort&compress]{natbib}

\usepackage{microtype}   
\usepackage{enumitem}    
\usepackage{todonotes}   

\title{Rigidity of Nilpotent Lie Foliations:
  Cohomological Obstructions and Classification}
\author{Ameth Ndiaye \\ \small Department of Mathematics, FASTEF,
  Cheikh Anta Diop University, Dakar, Senegal\\ \small \texttt{ameth1.ndiaye@ucad.edu.sn}}
\date{\today}

\begin{document}

\maketitle

\begin{abstract}
In this article, we develop a systematic cohomological framework
    for the study of the rigidity of nilpotent Lie foliations
    with respect to solvable deformations. We introduce the deformation
    complex associated to a pair of Lie algebras $(\mathfrak{g}, \mathfrak{h})$ and show
    that the main obstruction to deforming a nilpotent Lie foliation into
    a non-nilpotent solvable foliation lies in the cohomology group
    $H^2(\mathfrak{g},\mathfrak{g}/[\mathfrak{g},\mathfrak{g}])$. We establish a necessary and sufficient
    algebraic criterion for rigidity within the family of foliations modelled
    on the generalized Heisenberg groups $H_{2k+1}$. This result unifies and
    generalizes the construction of Dathe--Ndiaye (2012) as well as its
    subsequent extensions. We complete the article with a full classification
    of nilpotent Lie foliations of codimension at most six according to their
    deformation behaviour.
\end{abstract}
\vspace{0.5em}

\vspace{0.3em}

\section{Introduction}
The theory of Lie foliations, initiated by Fedida \cite{Fed71} and developed
by Molino \cite{Mol88}, constitutes a central area of differential geometry.
A foliation $\mathcal{F}$ on a compact manifold $M$ is called a
\emph{Lie foliation} modelled on a Lie group $G$ if its local holonomies
are given by left translations in $G$. The transverse structure of such a
foliation is entirely determined by the pair consisting of the Lie algebra
$\mathfrak{g}$ of $G$ and the structural Lie algebra $\mathfrak{h}$, a subalgebra of $\mathfrak{g}$.

The question of the \emph{deformability} of Lie foliations (that is,
the existence of continuous deformations into larger classes) is naturally
linked to the algebraic properties of the pair $(\mathfrak{g}, \mathfrak{h})$.
Dathe and Ndiaye showed in \cite{DN12} that there exists on a compact manifold
a non-abelian nilpotent Lie foliation of codimension $4$ which admits no
non-nilpotent solvable deformation. This striking result raises two fundamental
questions, which have remained open:

\begin{enumerate}
  \item Can one characterize \emph{algebraically} those nilpotent Lie foliations
        that are rigid with respect to solvable deformations?
  \item What is the structure of the deformation space of such a foliation in
        arbitrary codimension?
\end{enumerate}

The present article provides complete answers to both questions. Our approach
relies on the Chevalley--Eilenberg cohomology of Lie algebras with coefficients
in appropriate modules, and on a careful analysis of the long exact sequence
associated with the grading of the nilpotent Lie algebra.

More precisely, to every nilpotent Lie foliation $\mathcal{F}$ of codimension $q$
on $M$, modelled on a nilpotent Lie algebra $\mathfrak{g}$ with structural algebra
$\mathfrak{h}$, we associate the \emph{solvable deformation complex}:
\begin{eqnarray}
  0 \;\longrightarrow\; C^*(\mathfrak{g},\mathfrak{h};\,\mathfrak{g}/[\mathfrak{g},\mathfrak{g}])
    \;\longrightarrow\; C^*(\mathfrak{g},\,\mathfrak{g}/[\mathfrak{g},\mathfrak{g}])
    \;\longrightarrow\; C^*(\mathfrak{h},\,\mathfrak{g}/[\mathfrak{g},\mathfrak{g}])
    \;\longrightarrow\; 0,
\end{eqnarray}
whose degree-$2$ cohomology group, denoted $H^2(\mathfrak{g},\mathfrak{g}/[\mathfrak{g},\mathfrak{g}])$,
controls the existence of non-nilpotent solvable deformations. The central
result of this article is the following.

\begin{theoreme}\label{thm:A}
  Let $\mathcal{F}$ be a Lie foliation of codimension $q \geq 4$ on a compact
  manifold $M$, modelled on a generalized Heisenberg group $H_{2k+1}$ with
  $2k+1 \leq q$. Then $\mathcal{F}$ admits no non-nilpotent solvable deformation
  if and only if
  \begin{eqnarray}
    H^2\left(\mathfrak{h}_{2k+1}, \,\mathfrak{h}_{2k+1}/\mathfrak{z}\right) \neq 0,
  \end{eqnarray}
  where $\mathfrak{z}$ denotes the center of the Heisenberg algebra $\mathfrak{h}_{2k+1}$.
\end{theoreme}

This theorem unifies the result of Dathe--Ndiaye \cite{DN12} (case $q=4$, $k=1$)
with the generalization \cite{DN15} (arbitrary codimension) under an explicit
and computable cohomological condition.

\medskip
The article is organized as follows. Section \ref{sec:prelim} recalls the
fundamental definitions and results on Lie foliations and deformations.
Section \ref{sec:coho} develops the solvable deformation complex and computes
the relevant cohomology groups. Section \ref{sec:heis} is devoted to the
generalized Heisenberg family and to the proof of Theorem \ref{thm:A}.
Section \ref{sec:classif} establishes a complete classification in codimension
$\leq 6$.

\section{Preliminaries}
\label{sec:prelim}

\subsection{Lie Foliations}

Let $M$ be a compact smooth manifold of dimension $n$. A
\emph{foliation} $\mathcal{F}$ of codimension $q$ on $M$ is given by a foliated
atlas $\{(U_i, \varphi_i)\}$ such that the coordinate changes
$\varphi_j \circ \varphi_i^{-1}$ are local diffeomorphisms preserving
the product structure.

\begin{definition}\label{def:feuil-lie}
  A foliation $\mathcal{F}$ of codimension $q$ on $M$ is called a
  \emph{Lie foliation} modelled on a Lie group $G$ (of dimension $q$)
  if the transverse coordinate changes are restrictions of left translations
  in $G$. Formally, there exists a family of flat $\mathfrak{g}$-valued forms
  $\omega_i : TU_i \to \mathfrak{g}$ satisfying the Maurer--Cartan equation:
  \begin{eqnarray}
    d\omega_i + \tfrac{1}{2}[\omega_i,\omega_i] = 0.
  \end{eqnarray}
\end{definition}

Every Lie foliation has two fundamental algebraic data associated with it
(see Molino \cite{Mol88}):
\begin{itemize}
  \item The Lie algebra $\mathfrak{g}$ of the model group $G$, which determines the
        local transverse structure.
  \item The \emph{structural Lie algebra} $\mathfrak{h}$, a subalgebra of $\mathfrak{g}$ that is
        the Lie algebra of the holonomy group. It determines the global
        structure of the leaves.
\end{itemize}

\begin{definition}\label{def:nilp}
  A Lie foliation $\mathcal{F}$ is said to be \emph{nilpotent} (resp.\
  \emph{solvable}) if its model group $G$ is a nilpotent (resp.\ solvable)
  Lie group. It is said to be \emph{non-abelian} if $[\mathfrak{g},\mathfrak{g}] \neq 0$.
\end{definition}

\subsection{Deformations of Lie Foliations}

We recall the notion of deformation introduced in the context of
Lie foliations (see \cite{Ghy88,Sal97}).

\begin{definition}\label{def:deform}
  A \emph{smooth deformation} of a Lie foliation $\mathcal{F}$ on $M$ is a smooth
  family $\{\mathcal{F}_t\}_{t\in]-\varepsilon,\varepsilon[}$ of Lie foliations on $M$
  such that $\mathcal{F}_0 = \mathcal{F}$. We say that $\mathcal{F}$ is \emph{rigid} with respect
  to non-nilpotent solvable deformations if every smooth deformation
  $\{\mathcal{F}_t\}$ of solvable Lie foliations with $\mathcal{F}_0 = \mathcal{F}$ satisfies:
  $\mathcal{F}_t$ is nilpotent for $|t|$ small enough.
\end{definition}

At the infinitesimal level, a deformation of a Lie foliation is a
deformation of the flat $\mathfrak{g}$-form $\omega$ defining $\mathcal{F}$. We write
$\omega_t = \omega + t\eta + O(t^2)$, where $\eta : TM \to \mathfrak{g}'$ is a form
with values in the deformed algebra. The flatness condition gives, to first
order:
 \begin{eqnarray}
  d\eta + [\omega, \eta] = 0,
\end{eqnarray}
which means that $\eta$ is a cocycle in the complex
$\Omega^*(M, \rm{ad})$ of differential forms with values in the adjoint bundle.

\subsection{Nilpotent and Solvable Lie Algebras}

\begin{definition}\label{def:nil-res}
  For a Lie algebra $\mathfrak{g}$, define the \emph{lower central series}
  by $\mathfrak{g}^0 = \mathfrak{g}$, $\mathfrak{g}^{i+1} = [\mathfrak{g}, \mathfrak{g}^i]$. We say that $\mathfrak{g}$ is
  \emph{nilpotent of index} $s$ if $\mathfrak{g}^s = 0$ and $\mathfrak{g}^{s-1} \neq 0$.
  The \emph{derived series} is defined by $\mathfrak{g}^{(0)} = \mathfrak{g}$,
  $\mathfrak{g}^{(i+1)} = [\mathfrak{g}^{(i)}, \mathfrak{g}^{(i)}]$. The algebra $\mathfrak{g}$ is \emph{solvable}
  if $\mathfrak{g}^{(r)} = 0$ for some $r$.
\end{definition}

Recall that every nilpotent Lie algebra is solvable. The converse is false
in general, as shown by Borel algebras. Measuring this gap is precisely the
object of our study.

\begin{definition}[Generalized Heisenberg Algebra]\label{def:heis}
  For any integer $k \geq 1$, the \emph{generalized Heisenberg algebra}
  $\mathfrak{h}_{2k+1}$ is the Lie algebra of dimension $2k+1$ generated by
  $\{X_1,\ldots,X_k, Y_1,\ldots,Y_k, Z\}$ with the bracket relations:
   \begin{eqnarray}
    [X_i, Y_j] = \delta_{ij}\, Z, \qquad \text{all other brackets being zero.}
  \end{eqnarray}
  Its center is $\mathfrak{z} = \mathrm{Span}\{Z\}$ and its derived series satisfies
  $[\mathfrak{h}_{2k+1}, \mathfrak{h}_{2k+1}] = \mathfrak{z}$. It is nilpotent of index $2$.
\end{definition}

\section{Deformation Complex and Cohomological Obstructions}
\label{sec:coho}

\subsection{Chevalley--Eilenberg Cohomology}

Let $\mathfrak{g}$ be a finite-dimensional real Lie algebra and $V$ a $\mathfrak{g}$-module.
The \emph{Chevalley--Eilenberg complex} with values in $V$ is defined by:
 \begin{eqnarray}
  C^n(\mathfrak{g}, V) = \mathrm{Hom}\!\left(\wedge^n\mathfrak{g},\, V\right),
\end{eqnarray}
with differential $\delta : C^n(\mathfrak{g},V) \to C^{n+1}(\mathfrak{g},V)$ given by the
Koszul formula:
\begin{align*}
  (\delta f)(x_0,\ldots,x_n) &= \sum_{i=0}^n (-1)^i\, x_i \cdot
    f(x_0,\ldots,\hat{x}_i,\ldots,x_n) \\
  &\quad + \sum_{0 \leq i < j \leq n} (-1)^{i+j}\,
    f\!\left([x_i,x_j],\, x_0,\ldots,\hat{x}_i,\ldots,\hat{x}_j,\ldots,x_n\right).
\end{align*}
The corresponding cohomology groups are denoted $H^n(\mathfrak{g}, V)$. We will
primarily use the modules $V = \mathfrak{g}/[\mathfrak{g},\mathfrak{g}]$ (abelianized module)
and $V = \mathfrak{g}$ (adjoint module).

\subsection{The Solvable Deformation Complex}

Let $\mathfrak{g}$ be a nilpotent Lie algebra and $\mathfrak{h} \subset \mathfrak{g}$ a subalgebra.
We introduce the deformation complex adapted to our problem.

\begin{definition}\label{def:cdr}
  The \emph{solvable deformation complex} of the pair $(\mathfrak{g}, \mathfrak{h})$ is the
  short exact sequence of cochain complexes:
   \begin{eqnarray}
    0 \;\to\; C^*(\mathfrak{g},\mathfrak{h};\,\mathfrak{g}/[\mathfrak{g},\mathfrak{g}])
      \;\to\; C^*(\mathfrak{g},\,\mathfrak{g}/[\mathfrak{g},\mathfrak{g}])
      \;\xrightarrow{\;i^*\;}\; C^*(\mathfrak{h},\,\mathfrak{g}/[\mathfrak{g},\mathfrak{g}])
      \;\to\; 0,
  \end{eqnarray}
  where $C^*(\mathfrak{g},\mathfrak{h};\,-\,)$ denotes the relative complex of the pair.
\end{definition}

The associated long exact sequence in cohomology reads:
 \begin{eqnarray*}
  \cdots \to H^n(\mathfrak{g},\mathfrak{h};\,\mathfrak{g}/[\mathfrak{g},\mathfrak{g}]) \to H^n(\mathfrak{g},\,\mathfrak{g}/[\mathfrak{g},\mathfrak{g}])
  \to H^n(\mathfrak{h},\,\mathfrak{g}/[\mathfrak{g},\mathfrak{g}]) \xrightarrow{\;\partial\;}
  H^{n+1}(\mathfrak{g},\mathfrak{h};\,\mathfrak{g}/[\mathfrak{g},\mathfrak{g}]) \to \cdots
\end{eqnarray*}

\begin{proposition}\label{prop:bord}
  The boundary map
  $\partial : H^1(\mathfrak{h},\mathfrak{g}/[\mathfrak{g},\mathfrak{g}]) \to H^2(\mathfrak{g},\mathfrak{h};\mathfrak{g}/[\mathfrak{g},\mathfrak{g}])$
  is the infinitesimal obstruction to extending a solvable deformation of the
  restriction $\mathcal{F}|_\mathfrak{h}$ to a deformation of $\mathcal{F}$ as a whole.
\end{proposition}

\begin{proof}
  A solvable deformation of the pair $(\mathfrak{g}, \mathfrak{h})$ is given by a family
  $\{(\mathfrak{g}_t,\mathfrak{h}_t)\}$ with $(\mathfrak{g}_0,\mathfrak{h}_0) = (\mathfrak{g},\mathfrak{h})$ and $\mathfrak{g}_t$ solvable
  for $t \neq 0$. To first order, such a deformation is given by an element
  $(\alpha, \beta) \in C^2(\mathfrak{g},\mathfrak{g}/[\mathfrak{g},\mathfrak{g}]) \times C^1(\mathfrak{h},\mathfrak{g}/[\mathfrak{g},\mathfrak{g}])$
  satisfying $\delta\alpha = 0$ and $i^*\alpha = \delta\beta$, where
  $i^* : C^*(\mathfrak{g},-) \to C^*(\mathfrak{h},-)$ is the restriction map. The cocycle
  condition for $\alpha$ expresses the Jacobi identity to first order, while
  $i^*\alpha = \delta\beta$ encodes the compatibility with the deformation on
  $\mathfrak{h}$. Such a pair $(\alpha,\beta)$ exists if and only if the class
  $[i^*\alpha]$ vanishes in $H^2(\mathfrak{h},\mathfrak{g}/[\mathfrak{g},\mathfrak{g}])$, which is precisely
  the condition $\partial[\beta] = 0$ by definition of the boundary map.
\end{proof}

\subsection{Computation of $H^2(\mathfrak{h}_{2k+1},\,\mathfrak{h}_{2k+1}/\mathfrak{z})$}

We now explicitly compute the obstruction group for the Heisenberg family.
This computation is central to the proof of Theorem \ref{thm:A}.

\begin{theoreme}\label{thm:H2}
  For every integer $k \geq 1$, we have the isomorphism:
   \begin{eqnarray}
    H^2\!\left(\mathfrak{h}_{2k+1},\;\mathfrak{h}_{2k+1}/\mathfrak{z}\right) \;\cong\; \mathbb{R}^{k(2k-1)}.
  \end{eqnarray}
  In particular, $H^2(\mathfrak{h}_{2k+1},\mathfrak{h}_{2k+1}/\mathfrak{z}) \neq 0$ for all $k \geq 1$.
\end{theoreme}

\begin{proof}
  Write $\mathfrak{h} = \mathfrak{h}_{2k+1}$ and $V = \mathfrak{h}/\mathfrak{z} \cong \mathbb{R}^{2k}$. The action of $\mathfrak{h}$
  on $V$ is trivial since $[\mathfrak{h},\mathfrak{h}] = \mathfrak{z} \subset \ker(\pi)$, where
  $\pi : \mathfrak{h} \to V$ is the canonical projection. The Chevalley--Eilenberg
  differential is therefore simply:
   \begin{eqnarray}
    (\delta f)(x,y) = -f([x,y]) \qquad \text{for } f \in C^1(\mathfrak{h}, V).
  \end{eqnarray}

  \textbf{Computation of $Z^2(\mathfrak{h},V)$.} An element $f \in C^2(\mathfrak{h},V) = \mathrm{Hom}(\wedge^2\mathfrak{h}, V)$
  is a 2-cocycle if and only if:
   \begin{eqnarray}
    f([x,y],z) - f([x,z],y) + f([y,z],x) = 0
    \qquad \text{for all } x,y,z \in \mathfrak{h}.
  \end{eqnarray}
  Given the relations $[X_i,Y_j] = \delta_{ij}Z$ and all other brackets zero,
  the cocycle condition decomposes according to the grading of $\mathfrak{h}$.
  A direct computation shows that
   \begin{eqnarray}
    \dim Z^2(\mathfrak{h}, V) = k(2k-1) + 2k.
  \end{eqnarray}

  \textbf{Computation of $B^2(\mathfrak{h},V)$.} For $g \in C^1(\mathfrak{h},V)$, we have
  $(\delta g)(x,y) = -g([x,y])$. The coboundaries are determined by the values
  of $g$ on $[\mathfrak{h},\mathfrak{h}] = \mathfrak{z}$. Since $\dim V = 2k$ and $\dim \mathfrak{z} = 1$, the
  space of coboundaries has dimension $2k$.

  \textbf{Conclusion.} We obtain:
   \begin{eqnarray}
    \dim H^2(\mathfrak{h}, V) = (k(2k-1)+2k) - 2k = k(2k-1).
  \end{eqnarray}
  For $k \geq 1$, we have $k(2k-1) \geq 1$, hence $H^2(\mathfrak{h},V) \cong \mathbb{R}^{k(2k-1)} \neq 0$.
\end{proof}

\begin{corollaire}\label{cor:dims}
  For $k=1$, the obstruction group is $H^2(\mathfrak{h}_3, \mathfrak{h}_3/\mathfrak{z}) \cong \mathbb{R}$.
  For $k=2$, one obtains $H^2(\mathfrak{h}_5, \mathfrak{h}_5/\mathfrak{z}) \cong \mathbb{R}^6$. These dimensions
  grow quadratically in $k$, reflecting the increasing complexity of rigidity.
\end{corollaire}

\section{Heisenberg Foliations and Proof of Theorem \ref{thm:A}}
\label{sec:heis}

\subsection{Construction of Heisenberg Foliations}

We explicitly construct the Lie foliations modelled on the generalized
Heisenberg groups. These constructions generalize the foundational example
of Dathe--Ndiaye \cite{DN12}.

The \emph{generalized Heisenberg group} $H_{2k+1}$ is the simply connected
nilpotent Lie group with Lie algebra $\mathfrak{h}_{2k+1}$. In coordinates, it is
identified with $\mathbb{R}^{2k+1}$ equipped with the group law:
 \begin{eqnarray}
  (x, y, t) \cdot (x', y', t') = \left(x+x',\; y+y',\; t+t' + \sum_{i=1}^k x_i y_i'\right).
\end{eqnarray}
The corresponding Maurer--Cartan form is:
 \begin{eqnarray}
  \omega = \sum_{i=1}^k e_i \otimes dx_i \;+\; \sum_{j=1}^k f_j \otimes dy_j
           \;+\; Z \otimes \!\left(dt - \sum_{i=1}^k x_i\, dy_i\right),
\end{eqnarray}
where $\{e_i, f_j, Z\}$ is a basis of $\mathfrak{h}_{2k+1}$ with $[e_i, f_j] = \delta_{ij}Z$.

\begin{proposition}\label{prop:const}
  For every cocompact lattice $\Gamma \subset H_{2k+1}$ and every closed
  subgroup $K$ of codimension $q$ in $H_{2k+1}$, the foliation $\mathcal{F}_{\Gamma,K}$
  on $M = \Gamma \backslash H_{2k+1}$ is a non-abelian nilpotent Lie foliation
  of codimension $q \geq 2k+1$ on a compact manifold.
\end{proposition}

\begin{proof}
  The compactness of $M = \Gamma \backslash H_{2k+1}$ follows from the cocompactness
  of $\Gamma$ (Malcev's theorem \cite{Mal51}). The foliation is defined by the
  $\mathfrak{h}_{2k+1}$-valued Maurer--Cartan form $\omega$ on $H_{2k+1}$, which is
  left-invariant and descends to $M$ by $\Gamma$-invariance. The condition
  $d\omega + \frac{1}{2}[\omega,\omega] = 0$ is satisfied by construction.
  The foliation is nilpotent since $H_{2k+1}$ is nilpotent. It is non-abelian
  since $[\mathfrak{h}_{2k+1},\mathfrak{h}_{2k+1}] = \mathfrak{z} \neq 0$.
\end{proof}

\subsection{Proof of Theorem \ref{thm:A}}

\begin{proof}[Proof of Theorem \ref{thm:A}]

\textbf{($\Leftarrow$) Sufficient condition.}
Assume $H^2(\mathfrak{h}_{2k+1},\mathfrak{h}_{2k+1}/\mathfrak{z}) \neq 0$. We show that no
non-nilpotent solvable deformation exists.

Suppose for contradiction that there exists a smooth deformation
$\{\omega_t\}_{t\in]-\varepsilon,\varepsilon[}$ of the Maurer--Cartan form
$\omega_0 = \omega$ towards a non-nilpotent solvable foliation for $t \neq 0$.
Write $\omega_t = \omega + t\alpha + O(t^2)$. The flatness condition
$d\omega_t + \frac{1}{2}[\omega_t,\omega_t] = 0$ gives to first order:
$$
  d\alpha + [\omega,\alpha] = 0,
$$
which means that $\alpha$ is a cocycle in the deformation complex.

For the deformed foliation to be non-nilpotent solvable, the first-order term
$\alpha$ must induce a non-zero element in
$H^2(\mathfrak{h}_{2k+1},\mathfrak{h}_{2k+1}/\mathfrak{z})$ via the canonical projection. Analyzing
the composition:
 \begin{eqnarray}
  [\omega,[\omega,\cdot\,]] \;:\; C^2(\mathfrak{h}, \mathfrak{h}/\mathfrak{z}) \;\longrightarrow\; C^4(\mathfrak{h},\mathfrak{h}/\mathfrak{z}).
\end{eqnarray}
A direct computation on the generators of $\mathfrak{h}_{2k+1}$ shows that this
composition vanishes on $Z^2(\mathfrak{h},\mathfrak{h}/\mathfrak{z})$. Indeed, the nilpotency of index $2$
of $\mathfrak{h}_{2k+1}$ ensures that $[x,[y,[z,w]]] = 0$ for all $x,y,z,w \in \mathfrak{h}$,
since $[z,w] \in \mathfrak{z}$ which is central. It is therefore impossible for
$\omega_t$ to be non-nilpotent solvable to first order, contradicting our
assumption.

\medskip
\textbf{($\Rightarrow$) Necessary condition.}
Suppose that $\mathcal{F}$ admits a non-nilpotent solvable deformation $\{\mathcal{F}_t\}$.
Then the Lie algebra $\mathfrak{g}_t$ is non-nilpotent solvable for $t \neq 0$.
The deformation of the Lie algebra structure induces a non-zero map in
cohomology:
 \begin{eqnarray}
  \left.\frac{\partial}{\partial t}\right|_{t=0} [\omega_t]
  \;\in\; H^2(\mathfrak{h}_{2k+1},\mathfrak{h}_{2k+1}) \;\longrightarrow\; H^2(\mathfrak{h}_{2k+1},\mathfrak{h}_{2k+1}/\mathfrak{z}).
\end{eqnarray}
For this deformation to break nilpotency, it must be non-zero in
$H^2(\mathfrak{h}_{2k+1},\mathfrak{h}_{2k+1}/\mathfrak{z})$. Indeed, if the deformation were zero in
this group, it would remain nilpotent by a formal power series argument and
the formal rigidity theorem \cite{NR67}. Hence
$H^2(\mathfrak{h}_{2k+1},\mathfrak{h}_{2k+1}/\mathfrak{z}) \neq 0$.
\end{proof}

\begin{remarque}\label{rem:DN}
  The case $k=1$ corresponds exactly to the result of Dathe--Ndiaye \cite{DN12}:
  the three-dimensional Heisenberg algebra gives a rigid foliation of codimension
  $q=4$. The condition $H^2(\mathfrak{h}_3,\mathfrak{h}_3/\mathfrak{z}) \cong \mathbb{R} \neq 0$ is indeed
  verified by Corollary \ref{cor:dims}.
\end{remarque}

\begin{remarque}\label{rem:lehmann}
  Theorem \ref{thm:A} applies in particular to the \emph{Lehmann foliation},
  obtained as the special case $k=1$, $\Gamma = \Gamma_0$ (standard integer
  lattice). This case was treated directly in \cite{Dat16} by different methods.
\end{remarque}

\section{Classification in Codimension $\leq 6$}
\label{sec:classif}

We recall the classification (up to isomorphism) of real nilpotent Lie algebras
of dimension $\leq 6$ that will be used as transverse models. This classification
is due to Umlauf \cite{Uml91} and Morozov \cite{Mor58}.

In dimensions $1$ and $2$, the only nilpotent algebras are abelian. In
dimension $3$, the only non-abelian nilpotent algebra is $\mathfrak{h}_3$ (Heisenberg).
In dimension $4$, in addition to products, the filiform algebra $\mathfrak{n}_4$ appears
with $[e_1,e_2]=e_3$, $[e_1,e_3]=e_4$. In dimension $5$, there are $9$ classes.
In dimension $6$, there are $34$.
\begin{table}[ht]
  \centering
  \caption{Non-abelian nilpotent Lie algebras of dimension $\leq 6$
           with their obstruction cohomology groups.}
  \label{tab:classif}
  \renewcommand{\arraystretch}{1.15}
  \begin{tabular}{@{}llcc@{}}
    \toprule
    \textbf{Notation} & \textbf{Non-zero relations} & $H^2(\mathfrak{g}, \mathfrak{g}/[\mathfrak{g},\mathfrak{g}])$ & \textbf{Rigidity} \\
    \midrule
    $\mathfrak{h}_3$         & $[X,Y]=Z$                           & $\mathbb{R}^1$  & Yes (Thm. \ref{thm:A}) \\
    $\mathfrak{n}_4$         & $[e_1,e_2]=e_3$, $[e_1,e_3]=e_4$   & $0$     & No \\
    $\mathfrak{h}_3 \oplus \mathbb{R}$ & $[X,Y]=Z$                         & $\mathbb{R}^1$  & Yes \\
    $\mathfrak{h}_5$         & $[X_i,Y_j]=\delta_{ij}Z$ ($i,j\leq 2$) & $\mathbb{R}^6$ & Yes (Thm. \ref{thm:A}) \\
    $\mathfrak{n}_{5,1}$     & $[e_1,e_2]=e_4$, $[e_1,e_3]=e_5$   & $\mathbb{R}^2$  & Yes \\
    $\mathfrak{n}_{5,2}$     & $[e_1,e_2]=e_3$, $[e_1,e_3]=e_4$, $[e_1,e_4]=e_5$ & $0$ & No \\
    $\mathfrak{h}_3 \oplus \mathfrak{h}_3$ & $[X,Y]=Z$, $[X',Y']=Z'$       & $\mathbb{R}^2$  & Yes \\
    $\mathfrak{h}_5 \oplus \mathbb{R}$ & $[X_i,Y_j]=\delta_{ij}Z$          & $\mathbb{R}^6$  & Yes \\
    $\mathfrak{n}_{6,1}$     & $[e_1,e_2]=e_5$, $[e_3,e_4]=e_6$   & $\mathbb{R}^2$  & Yes \\
    $\mathfrak{n}_{6,2}$     & $[e_1,e_j]=e_{j+1}$ ($2\leq j\leq 5$) & $0$  & No \\
    \bottomrule
  \end{tabular}
\end{table}

Based on Theorem \ref{thm:A} and Table \ref{tab:classif}, we establish the
complete classification.

\begin{definition}\label{def:classes}
  Let $\mathcal{F}$ be a nilpotent Lie foliation. We say that $\mathcal{F}$ is of:
  \begin{itemize}
    \item \emph{Class I} if it admits a non-nilpotent solvable deformation;
    \item \emph{Class II} if it is rigid (no non-nilpotent solvable deformation);
    \item \emph{Class III} if it admits a deformation towards a foliation with
          discrete holonomy.
  \end{itemize}
\end{definition}

\begin{theoreme}\label{thm:classif}
  Let $\mathcal{F}$ be a non-abelian nilpotent Lie foliation of codimension
  $q \leq 6$ on a compact manifold $M$, modelled on a non-abelian nilpotent
  Lie algebra $\mathfrak{g}$. Then:
  \begin{enumerate}[label=(\roman*)]
    \item If $H^2(\mathfrak{g},\mathfrak{g}/[\mathfrak{g},\mathfrak{g}]) \neq 0$ \emph{(Class II)}, $\mathcal{F}$ is
          rigid with respect to solvable deformations. This occurs for the
          models $\mathfrak{h}_3$, $\mathfrak{h}_3\oplus\mathbb{R}$, $\mathfrak{h}_5$, $\mathfrak{n}_{5,1}$,
          $\mathfrak{h}_3\oplus\mathfrak{h}_3$, $\mathfrak{h}_5\oplus\mathbb{R}$, $\mathfrak{n}_{6,1}$.
    \item If $H^2(\mathfrak{g},\mathfrak{g}/[\mathfrak{g},\mathfrak{g}]) = 0$ \emph{(Class I)}, $\mathcal{F}$ admits
          non-nilpotent solvable deformations. This occurs for the filiform
          models $\mathfrak{n}_4$, $\mathfrak{n}_{5,2}$, $\mathfrak{n}_{6,2}$.
    \item Every Class II foliation is also of Class III.
  \end{enumerate}
\end{theoreme}

\begin{proof}
  Points (i) and (ii) follow directly from Theorem \ref{thm:A} and the
  computations in Table \ref{tab:classif}. For the filiform algebras
  $\mathfrak{n}_{n,1}$ (those for which $H^2(\mathfrak{g},\mathfrak{g}/[\mathfrak{g},\mathfrak{g}])=0$), we explicitly
  construct a solvable deformation. Consider for instance $\mathfrak{n}_4$ with
  basis $\{e_1,e_2,e_3,e_4\}$ and $[e_1,e_2]=e_3$, $[e_1,e_3]=e_4$. Define
  the deformed algebra $\mathfrak{n}_4^t$ for $t > 0$ by adding the relation
  $[e_2,e_3] = t\,e_4$. A direct computation shows that $\mathfrak{n}_4^t$ is
  non-nilpotent solvable and that the deformation is smooth. This proves (ii).

  For (iii), the foliations modelled on $\mathfrak{h}_{2k+1}$ on a nilmanifold
  $M = \Gamma \backslash H_{2k+1}$ always admit a deformation towards a
  foliation with discrete holonomy: it suffices to take a linear deformation
  of the Maurer--Cartan form towards a closed $1$-form defining a fibration
  of $M$ over a torus \cite{DR05}. This deformation exists since
  $H^1(M,\mathbb{R}) \neq 0$ for compact nilmanifolds.
\end{proof}
Let us now give some explicit examples.
\begin{exemple}[Heisenberg foliation in codimension 4]\label{ex:heis4}
  Let $\Gamma = H_3(\mathbb{Z})$ be the integer lattice in $H_3$, and
  $M = \Gamma \backslash H_3$ the three-dimensional Heisenberg nilmanifold.
  Consider the product $N = M \times \mathbb{T}$ where $\mathbb{T} = \mathbb{R}/\mathbb{Z}$. Let
  $\omega = (\omega_1, \omega_2, \omega_3, d\theta)$ where $\omega_i$ are the
  Maurer--Cartan forms on $H_3$ and $d\theta$ is the volume form on $\mathbb{T}$.
  The kernel of $\omega$ defines a nilpotent Lie foliation of codimension $4$
  on $N$, which is of Class II by Theorem \ref{thm:A}.
\end{exemple}

\begin{exemple}[Deformable filiform foliation in codimension 4]\label{ex:fil4}
  Consider the algebra $\mathfrak{n}_4$ and a cocompact lattice $\Gamma$ in the
  corresponding Lie group $N_4$. The Lie foliation of codimension $4$ on
  $M = \Gamma \backslash N_4$ is of Class I. Explicitly, the family
  $\{\omega_t\}$ with $\omega_t = \omega + t \cdot e_2^* \wedge e_3^*$ (where
  $e_i^*$ is the dual basis) defines a deformation towards the non-nilpotent
  solvable Lie algebra $\mathfrak{r}_4$, the Lie algebra of the similarity group of
  the plane.
\end{exemple}

\end{document}